\documentclass[12pt,reqno]{amsart}

\usepackage[all]{xy}

\usepackage{amssymb}

\newtheorem{theorem}{Theorem}[section]
\newtheorem{proposition}[theorem]{Proposition}
\newtheorem{lemma}[theorem]{Lemma}

\theoremstyle{definition}

\newtheorem{remark}[theorem]{Remark}

\numberwithin{equation}{section}

\begin{document}
\baselineskip=15.5pt

\title[Monodromy map for logarithmic differential systems]{On the monodromy 
map for the logarithmic differential systems}

\author[M. Aprodu]{Marian Aprodu}

\address{Faculty of Mathematics and Computer Science, University of Bucharesy \& ``Simion Stoilow'' Institute of Mathematics of the Romanian Academy,P.O. Box 1-764, 
014700 Bucharest, Romania}

\email{marian.aprodu@fmi.unibuc.ro \& marian.aprodu@imar.ro}

\author[I. Biswas]{Indranil Biswas}

\address{School of Mathematics, Tata Institute of Fundamental
Research, Homi Bhabha Road, Mumbai 400005, India}

\email{indranil@math.tifr.res.in}

\author[S. Dumitrescu]{Sorin Dumitrescu}

\address{Universit\'e C\^ote d'Azur, CNRS, LJAD, France}

\email{dumitres@unice.fr}

\author[S. Heller]{Sebastian Heller}

\address{Institut f\"ur Differentialgeometrie, Leibniz-Universit\"at Hannover, Germany}

\email{sheller@math.uni-hannover.de}



\date{}

\begin{abstract}
We study the monodromy map for logarithmic $\mathfrak g$-differential systems over an oriented surface $S_0$ of
genus $g$, with $\mathfrak g$ being
the Lie algebra of a complex reductive affine algebraic group $G$. These logarithmic 
$\mathfrak g$-differential systems are triples of the form $(X,\, D,\, \Phi)$, where $(X,\, D)\, \in\, {\mathcal 
T}_{g,d}$ is an element of the Teichm\"uller space of complex structures on $S_0$ with $d\, \geq\,1$
ordered marked points $D\, \subset\, S_0\,=\, X$ and 
$\Phi$ is a logarithmic connection on the trivial holomorphic principal $G$-bundle $X \times G$ over $X$ whose polar part 
is contained in the divisor $D$. We prove that the monodromy
map from the space of logarithmic $\mathfrak g$-differential systems to the character variety of 
$G$-representations of the fundamental group of $S_0\setminus D$ is an 
immersion at the generic point, in the following two cases:
\begin{enumerate}
\item $g\, \geq\, 2$, $d\, \geq\, 1$, and $\dim_{\mathbb C}G \, \geq\, d+2$;

\item $g\,=\,1$ and $\dim_{\mathbb C}G \, \geq\, d$.
\end{enumerate}
The above monodromy map is nowhere an immersion in the following two cases:
\begin{enumerate}
\item $g\,=\,0$ and $d\, \geq\, 4$;

\item $g\,\geq\,1$ and $\dim_{\mathbb C}G \, <\, \frac{d+3g-3}{g}$.
\end{enumerate}
This extends to the logarithmic case the main results in \cite{CDHL}, \cite{BD} dealing with 
nonsingular holomorphic $\mathfrak g$-differential systems (which corresponds to the case of $d\,=\,0$).
\end{abstract}

\maketitle

\tableofcontents

\section{Introduction}

The study of the Riemann-Hilbert mapping, which associates to a flat (algebraic or holomorphic) connection its 
monodromy morphism from the fundamental group is a classical topic in algebraic and analytical geometry (see, for 
instance, \cite{De}, \cite{Ka} and references therein).

We recall the set-up and results of \cite{CDHL} and \cite{BD}, the predecessors of this paper.
Let $G$ be a connected reductive affine algebraic group defined over $\mathbb C$,
with $\dim G \, >\, 0$, and let $\mathfrak g$ be the
Lie algebra of $G$. A $\mathfrak g$--differential system is a pair of the
form $(X,\, \Phi)$, where $X$ is a complex structure on a compact oriented smooth surface $S_0$
of genus $g$, and $\Phi$ is a
holomorphic connection on the trivial holomorphic principal $G$--bundle $X \times G$ over the
Riemann surface $X$. A $\mathfrak g$--differential system $(X,\, \Phi)$ is called irreducible if $\Phi$
is not induced by a holomorphic connection on $X\times P$ for some proper parabolic subgroup $P$ of $G$. 
Since any holomorphic connection on a Riemann surface is flat, associating the monodromy
representation to a holomorphic connection we obtain a map from the space of irreducible
$\mathfrak g$--differential systems to the irreducible $G$-character variety ${\rm Hom}(\pi_1(S_0),\, G)^{\rm ir}/G$.
This monodromy map is actually holomorphic.

The main result of \cite{CDHL} says that, if $g=2$, this Riemann-Hilbert monodromy map is a local diffeomorphism 
from the space of irreducible $\mathfrak g$--differential systems into the irreducible $G$-character variety, for 
$G\,=\, \text{SL}(2, {\mathbb C})$. Being inspired by \cite{CDHL}, in \cite{BD} it was shown that, for all $g\, \geq\, 
2$, the above monodromy map is an immersion on an open dense subset of the space of irreducible $\mathfrak 
g$--differential systems, for all reductive groups $G$ with $\dim_{\mathbb C} G\, \geq\, 3$.

Our aim here is to study the Riemann-Hilbert monodromy mapping for logarithmic $\mathfrak g$--differential 
systems, where $\mathfrak g$ is as above. These logarithmic $\mathfrak g$--differential systems are defined by 
triples of the form $(X,\, D,\, \Phi)$, where $(X,\, D)\, \in\, {\mathcal T}_{g,d}$ is an element of the 
Teichm\"uller space of complex structures on $S_0$ with $d$ ordered marked points $D\, \subset\, S_0\,=\, X$ (see 
Section \ref{log connect}), and $\Phi$ is a logarithmic connection on the trivial holomorphic principal $G$-bundle 
$X \times G$ over $X$ whose polar part is contained in the divisor $D$.

We prove the following (see Theorem \ref{thm1}):

\begin{theorem}\label{thm-i}
Assume that $3g-3+d\, >\, 0$ and $d\, \geq\, 1$.
The Riemann-Hilbert monodromy mapping from the above space of irreducible logarithmic $\mathfrak
g$--differential systems to the character variety of irreducible $G$-representations of the fundamental
group of $S_0\setminus D$ is an immersion at the generic point in the following two cases:
\begin{enumerate}
\item $g\, \geq\, 2$ and $\dim_{\mathbb C} G\, \geq \, d+2$;

\item $g\, =\, 1$ and $\dim_{\mathbb C} G\, \geq \, d$.
\end{enumerate}

The Riemann-Hilbert monodromy mapping from the above space of irreducible logarithmic $\mathfrak
g$--differential systems to the character variety of irreducible $G$-representations of the fundamental
group of $S_0\setminus D$ is nowhere an immersion in the following two cases:
\begin{enumerate}
\item $g\, =\, 0$;

\item $g\, \geq\, 1$ and $\dim_{\mathbb C}G \, <\, \frac{d+3g-3}{g}$ (in particular,
when $g\,=\,1$ and $\dim_{\mathbb C}G \, <\, d$).
\end{enumerate}
\end{theorem}

We note that Theorem \ref{thm-i} gives a complete answer only when $g\,=\, 0$ or $g\,=\,1$. For given
$g\, \geq\, 2$ and $G$, there are finitely many cases of $d$ that are not addressed in Theorem \ref{thm-i}.
When $g\,=\,1$ and $d\,=\, 0$, from the first part of Theorem \ref{thm-i} it follows that the
monodromy mapping from the space of irreducible logarithmic $\mathfrak
g$--differential systems is an immersion at the generic point; see Remark \ref{gd0}.

Theorem \ref{thm-i}, extends to the class of logarithmic $\mathfrak g$-differential systems, the main result in 
\cite{BD} which deals with the nonsingular holomorphic $\mathfrak g$-differential systems (corresponding to the 
case $d\,=\,0$). Notice that the hypothesis $3g-3+d\, >\, 0$ in Theorem \ref{thm-i} implies that the above 
Teichm\"uller space ${\mathcal T}_{g,d}$ has positive dimension.

Given a reductive complex affine algebraic group $G_0$, by setting $G$ to be the product group
$G^m_0$, $m\, \geq\, 1$, we can make its dimension arbitrarily large.

The proof of Theorem \ref{thm-i} is based on a transversality result in the moduli space $\mathcal B_{G}$ of
quadruples of the form $(X,\, D,\, E_G,\, \Phi)$, where
\begin{itemize}
\item $(X,\, D)\, \in\, {\mathcal T}_{g,d}$,

\item $E_G$ is a holomorphic principal $G$--bundle on $X$ such that $E_G$ is topologically trivial,
and

\item $\Phi$ is a logarithmic connection on $E_G$ whose
polar part is contained in $D$.
\end{itemize}

A key ingredient of this transversality condition is proved in Lemma \ref{lem4} which is an adaptation to the 
logarithmic case of Theorem 1.1 in \cite{Gi} (where its proof is attributed to R. Lazarsfeld).

The article is organized as follows. Sections \ref{Atiyah bundle} and \ref{log connect} are preparatory: they 
introduce the concept of logarithmic connections on holomorphic principal bundles, the above moduli space $\mathcal 
B_{G}$ of quadruples $(X,\, D,\, E_G,\, \Phi)$ and the $G$-character variety. We describe the infinitesimal 
deformation space of quadruples (the tangent space of $\mathcal B_{G}$) as the first hypercohomology group of a 
certain $2$-term complex (see Proposition \ref{prop1}). Section \ref{monodromy map} is devoted to the 
proof of the main result (Theorem \ref{thm1}) and deals with the transversality, in the tangent space of $\mathcal 
B_{G}$, between the isomonodromy foliation and the subspace of logarithmic $\mathfrak g$--differential systems. 
This transversality condition, which is equivalent to the monodromy map being an immersion on the space 
logarithmic of $\mathfrak g$--differential systems, is proved by combining a criteria given in Lemma \ref{lem3}
(also Proposition \ref{prop3}), with Lemma \ref{lem4} (dealing with the case $g \geq 3$) and Lemma \ref{lem5} 
(dealing with the case of $g\,=\,2$).

\section{The logarithmic Atiyah bundle} \label{Atiyah bundle} 

Let $X$ be a compact connected Riemann surface. Let
\begin{equation}\label{e1}
D\, :=\, \{x_1,\, \cdots,\, x_d\}\, \subset\, X
\end{equation}
be $d$ distinct points, with $d\, \geq\, 2$. For notational convenience, the divisor
$x_1+\ldots+ x_d$ of degree $d$ on $X$ will also be denoted by $D$.
For a holomorphic vector bundle $V$ on $X$, the holomorphic vector bundles
$V\otimes {\mathcal O}_X(D)$ and $V\otimes {\mathcal O}_X(-D)$ will be denoted
by $V(D)$ and $V(-D)$ respectively. The holomorphic tangent and cotangent bundles
of $X$ will be denoted by $TX$ and $K_X$ respectively.

Let $G$ be a connected complex affine algebraic group with $\dim G \, >\, 0$. The Lie algebra of $G$ will be
denoted by $\mathfrak g$. Let
\begin{equation}\label{e2}
p\, :\, E_G\, \longrightarrow\, X
\end{equation}
be a holomorphic principal $G$--bundle over $X$. The action of $G$ on $E_G$ produces
an action of $G$ on the holomorphic tangent bundle $TE_G$ of $E_G$. The quotient
\begin{equation}\label{e3}
\text{At}(E_G)\, :=\, (TE_G)/G \, \longrightarrow\, X
\end{equation}
is the Atiyah bundle for $E_G$ \cite{At}. Let $dp\, :\, TE_G\,\longrightarrow\, p^*TX$
be the differential of the map $p$ in \eqref{e2}. Let
\begin{equation}\label{e3p}
\text{ad}(E_G)\, :=\, \text{kernel}(dp)/G \, \subset\, (TE_G)/G
\end{equation}
be the adjoint bundle for $E_G$. Note that this holomorphic vector bundle $\text{kernel}(dp)$ is
identified with the trivial holomorphic vector bundle $E_G\times{\mathfrak g}\,\longrightarrow\,
E_G$ using the action of $G$ on $E_G$. Hence $\text{ad}(E_G)$ coincides with the vector
bundle $E_G\times^G{\mathfrak g}\,\longrightarrow\, X$ associated to $E_G$ for
the adjoint action of $G$ on ${\mathfrak g}$.

Thus we have a short exact sequence of holomorphic vector bundles on $X$
\begin{equation}\label{e4}
0\, \longrightarrow\, \text{ad}(E_G) \, \longrightarrow\,\text{At}(E_G)
\, \stackrel{d'p}{\longrightarrow}\, TX \, \longrightarrow\, 0\, ,
\end{equation}
where $\text{At}(E_G)$ is defined in \eqref{e3}, and the projection $d'p$ is induced by
$dp$; the sequence in \eqref{e4} is known as the Atiyah exact sequence. Define
\begin{equation}\label{e5}
\text{At}(E_G)(-\log D)\, :=\, (d'p)^{-1}(TX(-D))\, \subset\, \text{At}(E_G)\, ,
\end{equation}
where $d'p$ is the homomorphism in \eqref{e4}. So, from \eqref{e4} we have the
\textit{logarithmic Atiyah exact sequence}
\begin{equation}\label{e6}
0\, \longrightarrow\, \text{ad}(E_G) \, \stackrel{\iota_0}{\longrightarrow}\,\text{At}(E_G)(-\log D)
\, \stackrel{\widehat{dp}}{\longrightarrow}\, TX(-D) \, \longrightarrow\, 0\, ,
\end{equation}
where $\widehat{dp}$ is the restriction of the homomorphism $d'p$ to $\text{At}(E_G)(-\log D)$, and
$\iota_0$ is given by the homomorphism $\text{ad}(E_G) \, \longrightarrow\,\text{At}(E_G)$
in \eqref{e4}. We have the following commutative diagram of homomorphisms 
\begin{equation}\label{e7}
\begin{matrix}
0 & \longrightarrow & \text{ad}(E_G) & \stackrel{\iota_0}{\longrightarrow} & \text{At}(E_G)(-\log D)
& \stackrel{\widehat{dp}}{\longrightarrow}& TX(-D) & \longrightarrow & 0\\
&& \Vert &&\,\,\,\,\Big\downarrow\iota &&\,\, \,\,\Big\downarrow\iota'\\
0 & \longrightarrow & \text{ad}(E_G) & \longrightarrow & \text{At}(E_G)
& \stackrel{d'p}{\longrightarrow}& TX & \longrightarrow & 0
\end{matrix}
\end{equation}
where $\iota$ and $\iota'$ are the natural inclusion maps.

A \textit{logarithmic connection} on $E_G$ with polar part in $D$ is a holomorphic homomorphism
$$
\Phi\, :\, TX(-D)\, \longrightarrow\,\text{At}(E_G)(-\log D)
$$
such that
\begin{equation}\label{e8}
\widehat{dp}\circ\Phi\, =\, \text{Id}_{TX(-D)}\, ,
\end{equation}
where $\widehat{dp}$ is the surjective homomorphism in \eqref{e6}.

Since we have $\iota'\circ\widehat{dp}\,=\, (d'p)\circ\iota$ (see \eqref{e7}), and $\iota' (y)(TX(-D)_y)\,=\, 0$
for every point $y\, \in\, D$ in \eqref{e1}, for a logarithmic connection
$\Phi$ on $E_G$, from \eqref{e8} we have
$$
\iota'\circ\widehat{dp}\circ\Phi(TX(-D)_y)\,=\,\iota' (y)(TX(-D)_y) \,=\, 0
$$
for every $y\, \in\, D$. Consequently, from the commutativity of \eqref{e7}
we conclude that $(d'p)\circ\iota\circ\Phi(TX(-D)_y)\,=\ 0$. This
implies that
\begin{equation}\label{e9}
\iota\circ\Phi(TX(-D)_y)\, \subset\, \text{ad}(E_G)_y\, \subset\, \text{At}(E_G)_y
\end{equation}
(see \eqref{e7}). On the other hand, $TX(-D)_y\,=\, {\mathbb C}$ by the Poincar\'e adjunction formula
\cite[p.~146]{GH}; for any holomorphic coordinate function $z$ on $X$ around $y$ with $z(y)\,=\, 0$, the map
${\mathbb C}\, \longrightarrow\, TX(-D)_y$ defined by $\lambda\, \longmapsto\, (\lambda\frac{dz}{z})(y)$
is actually independent of the choice of the coordinate function $z$. The element
$$
(\iota\circ\Phi)(y)(1)\, \in \, \text{ad}(E_G)_y
$$
(see \eqref{e9}) is called the \textit{residue} of $\Phi$ at $y$; see \cite{De}.

Fixing $X$, the infinitesimal deformations of the principal $G$--bundle $E_G$ are parametrized
by $H^1(X,\, \text{ad}(E_G))$ \cite{Do}.

We recall that the infinitesimal deformations of the $d$-pointed Riemann surface $(X,\, D)$ 
are parametrized by $H^1(X,\, TX(-D))$. The infinitesimal deformations of the above triple 
$(X,\, D,\, E_G)$ are parametrized by $H^1(X,\, \text{At}(E_G)(-\log D))$ \cite{BHH}, 
\cite{Ch1}, \cite{Ch2}, \cite{Hu}, \cite{Do}.

The following lemma is standard (see \cite[Section 2.2]{BHH} and \cite{Hu}).

\begin{lemma}\label{lem1}\mbox{}
\begin{enumerate}
\item The homomorphism of cohomologies $$\widehat{dp}_*\, :\, H^1(X,\, {\rm At}(E_G)(-\log D))\, \longrightarrow\,
H^1(X,\, TX(-D))\, ,$$ induced by the projection $\widehat{dp}$ in \eqref{e6}, corresponds to the forgetful
map from the infinitesimal deformations of the triple $(X,\, D,\, E_G)$ to the
infinitesimal deformations of the pair $(X,\, D)$ obtained by simply forgetting the principal $G$--bundle.

\item The homomorphism of cohomologies $$\iota_{0*} \, : \, H^1(X,\, {\rm ad}(E_G)))\, \longrightarrow\,H^1(X,\,
{\rm At}(E_G)(-\log D))\, ,$$
induced by the homomorphism $\iota_0$ in \eqref{e6}, coincides with the map from the 
infinitesimal deformations of the principal $G$--bundle $E_G$ to the
infinitesimal deformations of the triple $(X,\, D,\, E_G)$ obtained by keeping the pair $(X,\, D)$ fixed.
\end{enumerate}
\end{lemma}

\section{Logarithmic connections and isomonodromy}\label{log connect} 

\subsection{Logarithmic Atiyah bundle}

Since $\text{At}(E_G)\, :=\, (TE_G)/G$ (see \eqref{e3}), the subsheaf $\text{At}(E_G)(-\log D)\,
\subset\, \text{At}(E_G)$ corresponds to a subsheaf of the sheaf of $G$--invariant holomorphic
vector fields on $E_G$. We will have occasions to use the following description of this subsheaf
of the sheaf of $G$--invariant holomorphic vector fields on $E_G$.

Let $$\widetilde{D}\, :=\, p^{-1}(D) \, \subset\, E_G$$ be the divisor, where $p$ is the
projection in \eqref{e2}. Let
$$
TE_G(-\log \widetilde{D})\, \subset\, TE_G
$$
be the corresponding logarithmic tangent bundle. We recall that this subsheaf is characterized by
the following property: A holomorphic vector field $v$, defined on an open subset $U\,\subset\, E_G$,
is a section of $TE_G(-\log \widetilde{D})$ if and only if for every 
holomorphic function $f$ on $U$ that vanishes on $\widetilde{D}\bigcap U$, the function
$v(f)$ also vanishes on $\widetilde{D}\bigcap U$. Since the divisor $\widetilde{D}$ is smooth,
it follows that $TE_G(-\log \widetilde{D})$ is a locally
free ${\mathcal O}_{E_G}$--submodule of $TE_G$. Consequently, $TE_G(-\log \widetilde{D})$ is a
holomorphic vector bundle on $E_G$. The above characterizing property of $TE_G(-\log \widetilde{D})$
immediately implies that the Lie bracket operation of locally defined holomorphic vector fields on
$E_G$ preserves the subsheaf $TE_G(-\log \widetilde{D})$.

To describe $TE_G(-\log \widetilde{D})$ locally, take a point $x\, \in\, \widetilde{D}$. Let
$(z_1,\, z_2,\, \cdots,\, z_m)$ be holomorphic coordinate functions on $E_G$ defined around $x$
such that $z_1\,=\, p\circ z$ for some holomorphic coordinate function $z$ on $X$ around $p(x)$, and
also $z_i(x)\,=\, 0$ for all $1\, \leq\, i\, \leq\, m$; here
$p$ denotes the projection in \eqref{e2}. Then $TE_G(-\log \widetilde{D})$ around $x$
is generated by the holomorphic vector fields $z_1\frac{\partial}{\partial z_1},\,
\frac{\partial}{\partial z_2},\, \cdots,\, \frac{\partial}{\partial z_m}$.

The action of $G$ on $TE_G$, induced by the action of $G$ on $E_G$, actually preserves the
subsheaf $TE_G(-\log \widetilde{D})$. It is now straightforward to check that
\begin{equation}\label{e10}
\text{At}(E_G)(-\log D)\,=\, TE_G(-\log \widetilde{D})/G\, .
\end{equation}

Let
\begin{equation}\label{ep}
\Phi\, :\, TX(-D)\, \longrightarrow\,\text{At}(E_G)(-\log D)
\end{equation}
be a logarithmic connection
on $E_G$. Let
\begin{equation}\label{e11}
\widetilde{\Phi}\, :\, \text{At}(E_G)(-\log D)\, \longrightarrow\,\text{ad}(E_G)
\end{equation}
be the holomorphic homomorphism uniquely determined by the following conditions:
\begin{enumerate}
\item $\widetilde{\Phi}\circ\iota_0\,=\, \text{Id}_{\text{ad}(E_G)}$, where $\iota_0$ is
the injective homomorphism in \eqref{e6}, and

\item $\text{kernel}(\widetilde{\Phi})\,=\, \Phi(TX(-D))$.
\end{enumerate}
In view of \eqref{e3p} and \eqref{e10}, the homomorphism $\widetilde{\Phi}$ in \eqref{e11}
produces a $G$--invariant surjective holomorphic homomorphism
\begin{equation}\label{e12}
\Phi'_0\ :\, TE_G(-\log \widetilde{D})\, \longrightarrow\,\text{kernel}(dp)\, ,
\end{equation}
where $p$ is the projection in \eqref{e2}.

Let $w$ be a holomorphic vector field on an open subset $U\, \subset\, X$
that vanishes on $U\bigcap D$. In view of \eqref{e10},
the section $\Phi(w)$ of $\text{At}(E_G)(-\log D)\big\vert_U$ corresponds to a unique $G$--invariant
holomorphic section of $TE_G(-\log \widetilde{D})\big\vert_{p^{-1}(U)}$
satisfying the condition that $$dp(\Phi(w))\,=\, p^*w$$ (as sections of $p^*TX)$); let
\begin{equation}\label{k1}
\Phi(w)'\, \in\, H^0(p^{-1}(U),\, TE_G(-\log \widetilde{D})\big\vert_{p^{-1}(U)})
\end{equation}
denote this section constructed from $w$.

\begin{lemma}\label{lem2}
Let $v$ be a $G$--invariant
holomorphic section of $TE_G(-\log \widetilde{D})\big\vert_{p^{-1}(U)}$. Then the following three
hold:
\begin{enumerate}
\item The holomorphic section $\Phi'_0([\Phi(w)',\, v])$ of ${\rm kernel}(dp)$
is $G$--invariant, where $\Phi'_0$ is the homomorphism in \eqref{e12} and
$\Phi(w)'$ is the section of $TE_G(-\log \widetilde{D})\big\vert_{p^{-1}(U)}$ constructed above from $w$.

\item For every holomorphic function $h$ on $U$,
$$
\Phi'_0([\Phi(h\cdot w)',\, v])\,=\, (h\circ p)\cdot \Phi'_0([\Phi(w)',\, v])\, .
$$

\item If $v\,=\, \Phi(v_1)'$ for some holomorphic section $v_1$ of $T(-D)\big\vert_U$, then
$$
\Phi'_0([\Phi(w)',\, v])\,=\, 0\, .
$$
\end{enumerate}
\end{lemma}

\begin{proof}
As noted before, the Lie bracket operation of locally defined holomorphic vector fields on
$E_G$ preserves the subsheaf $TE_G(-\log \widetilde{D})$. Since the homomorphism $\Phi'_0$
in \eqref{e12} is $G$--invariant, and $\Phi(w)'$ is $G$--invariant, while $v$ is
given to be $G$--invariant, it follows that $\Phi'_0([\Phi(w)',\, v])$ is
also $G$--invariant.

To prove the second statement, consider the identity
$$
[\Phi(h\cdot w)',\, v]\,=\, (h\circ p)\cdot [\Phi(w)',\, v]
-v(h\circ p)\cdot \Phi(w)'\, .
$$
Since $\Phi'_0(\Phi(w)')\,=\, 0$, where $\Phi'_0$ is the homomorphism
in \eqref{e12}, the second statement follows from this identity.

The third statement follows from the fact that any holomorphic one-dimensional distribution is integrable.
\end{proof}

In view of \eqref{e3p} and \eqref{e10}, from Lemma \ref{lem2}(2) we get a homomorphism
$$
\widetilde{\Phi}\,\, :\,\,TX(-D)\otimes \text{At}(E_G)(-\log D)\, \longrightarrow\, \text{ad}(E_G)\, .
$$
Then the homomorphism
$$
\widetilde{\Phi}\otimes {\rm Id}_{TX(-D)^*}\, :\,\text{At}(E_G)(-\log D)TX(-D)\otimes TX(-D)^*
\, \longrightarrow\, \text{ad}(E_G)\otimes TX(-D)^*
$$
produces a homomorphism
\begin{equation}\label{e13}
\widehat{\Phi}
\, :\, \text{At}(E_G)(-\log D)\, \longrightarrow\, \text{ad}(E_G)\otimes TX(-D)^*\,=\,
\text{ad}(E_G)\otimes K_X(D)
\end{equation}
using the duality pairing $TX(-D)\otimes TX(-D)^*\, \longrightarrow\, {\mathcal O}_X$.

Let ${\mathcal C}_{\bullet}$ be the $2$--term complex of sheaves on $X$
\begin{equation}\label{e14}
{\mathcal C}_{\bullet}\,\, :\,\,
{\mathcal C}_0\,:=\, \text{At}(E_G)(-\log D)\, \stackrel{\widehat{\Phi}}{\longrightarrow}\,
{\mathcal C}_1\,:=\, \text{ad}(E_G)\otimes K_X(D)\, ,
\end{equation}
where $\widehat{\Phi}$ is the homomorphism constructed in \eqref{e13}, and ${\mathcal C}_i$ is
at the $i$-th position.

{}From Lemma \ref{lem2}(3) we know that
$$
\widehat{\Phi}\circ\Phi\,=\, 0\, .
$$
Consequently, the logarithmic connection $\Phi$ in \eqref{ep} produces a homomorphism of complexes
\begin{equation}\label{e15}
\Phi^C\, :\, TX(-D)\, \longrightarrow\, {\mathcal C}_{\bullet}\, ,
\end{equation}
where $TX(-D)$ is the one-term complex concentrated at the $0$-th position, and
${\mathcal C}_{\bullet}$ is the complex in \eqref{e14}. In other words, we have the
commutative diagram
\begin{equation}\label{e15a}
\begin{matrix}
TX(-D) & : & TX(-D) & \longrightarrow & 0\\
\Phi^C\Big\downarrow\,\,\,\,\,\,\,\,
&& \,\,\,\, \,\Big\downarrow\Phi && \Big\downarrow\\
{\mathcal C}_\bullet &: & {\mathcal C}_0 & \stackrel{\widehat{\Phi}}{\longrightarrow} & {\mathcal C}_1
\end{matrix}
\end{equation}

\subsection{Character variety}

Let ${\mathcal T}_{g,d}$ denote the Teichm\"uller space of compact connected Riemann surfaces of genus $g$
with $d$ ordered marked points, where $d\, \geq\,1$. We will always assume that $3g-3+d\, >\, 0$. This
${\mathcal T}_{g,d}$ is a complex
manifold of dimension $3g-3+d$. We recall a description of ${\mathcal T}_{g,d}$ which will be used here.
Let $S_0$ be an oriented $C^\infty$ surface of genus $g$, and let $D_0\, \subset\, S_0$ be
$d$ ordered distinct points. Let $\mathbf{C}(S_0)$ denote the space of all $C^\infty$ complex structures on $S_0$
compatible with the given orientation of $S_0$. Let $\text{Diff}^0_{D_0}(S_0)$ denote the group of all orientation
preserving diffeomorphisms
$$
\beta\, :\, S_0\, \longrightarrow\, S_0
$$
such that
\begin{itemize}
\item $\beta(x)\,=\, x$ for every $x\, \in\, D_0$, and

\item $\beta$ is homotopic to the identity map
of $S_0$ through a continuous family of diffeomorphisms $\beta_t$ of $S_0$, $0\,
\leq\, t\, \leq\, 1$, such that $\beta_t(x)\,=\, x$ for all $t$ and all $x\, \in\, D_0$.
\end{itemize}
The group $\text{Diff}^0_{D_0}(S_0)$ acts on $\mathbf{C}(S_0)$ by
pushing forward complex structures using diffeomorphisms. Then we have
$$
{\mathcal T}_{g,d}\, =\, \mathbf{C}(S_0)/\text{Diff}^0_{D_0}(S_0)\, .
$$

We now assume the complex connected affine algebraic group $G$ to be reductive.
The complement $S_0\setminus D_0$ will be denoted by $S'_0$. Let
\begin{equation}\label{e16}
{\mathcal R}_G(S'_0)\, :=\, \text{Hom}^{\rm ir}(\pi_1(S'_0),\, G)/G
\end{equation}
be the irreducible $G$--character variety for $S'_0$; the space $\text{Hom}^{\rm ir}(\pi_1(S'_0),\, G)$ consists of
all homomorphisms $\gamma\, :\, \pi_1(S'_0)\, \longrightarrow\, G$ such that $\gamma(\pi_1(S'_0))$
is not contained in any proper parabolic subgroup of $G$.
We note that ${\mathcal R}_G(S'_0)$ does
not depend on the choice of the base point needed to define the fundamental group of $S'_0$. 
Since $\pi_1(S'_0)$ is finitely presented, the complex algebraic structure of $G$ produces a complex
algebraic structure on ${\mathcal R}_G(S'_0)$, so ${\mathcal R}_G(S'_0)$ is a complex affine variety.
It is in fact a smooth complex orbifold. We have
\begin{equation}\label{dr}
\dim {\mathcal R}_G(S'_0)\,=\, (2g+d-1)\cdot\dim_{\mathbb C} G - \dim_{\mathbb C} [G,\, G]\, .
\end{equation}

For more details of the above dimension count the reader is referred to \cite{Go} and \cite[Proposition 49]{Si} (to 
which the monodromy around the poles should be added).

\subsection{Monodromy of logarithmic connections}

A logarithmic connection $\Phi$ on a holomorphic principal $G$--bundle $E_G\, \longrightarrow\, X$
is called \textit{irreducible} if there is no holomorphic reduction of structure group $E_P\, \subset\,
E_G\big\vert_{X\setminus D}$ to some proper parabolic subgroup $P\, \subset\, G$,
over the open subset $X\setminus D\, \subset\, X$, such that $\Phi$ is induced by a holomorphic connection
on $E_P$.

The above definition of irreducibility needs a clarification, because in the 
special case where $E_G$ is the trivial holomorphic principal $G$--bundle,
and $D$ is the zero divisor --- so the logarithmic connection $\Phi$ is holomorphic, meaning
it has no poles --- this definition
of irreducibility is, a priori, weaker than the definition, given in the Introduction,
of irreducible holomorphic $\mathfrak{g}$--differential systems. More precisely, in the
definition, given in the Introduction, of irreducible holomorphic $\mathfrak{g}$--differential systems,
the principal $P$--bundle is required to be the trivial bundle $X\times P\, \longrightarrow\, X$, while the above
definition does not impose any other condition on $E_P$ apart from the condition that the logarithmic connection $\Phi$
is induced by a logarithmic connection on $E_P$. We will show the following:

Let $\Phi$ be a holomorphic connection on the trivial principal $G$--bundle $${\mathcal E}^0_G\, :=\, X\times G
\, \longrightarrow\, X\, ,$$ and let $E_P\, \subset\, {\mathcal E}^0_G$ be a holomorphic reduction of
structure group to $P$ over $X$, such that
$\Phi$ is induced by a holomorphic connection on $E_P$. Then $E_P$ is the trivial principal $P$--bundle
$X\times P\, \longrightarrow\, X$.

To prove the above statement, first note that a holomorphic reduction of structure group
$E_P\, \subset\, {\mathcal E}^0_G$ to $P$ is given by a holomorphic map
$\phi\, :\, X\, \longrightarrow\, G/P$. For this map $\phi$, we have
\begin{equation}\label{d1}
\phi^*T(G/P)\,=\, \text{ad}({\mathcal E}^0_G)/\text{ad}(E_P)\, .
\end{equation}
If $E_P$ admits a holomorphic connection $\Phi_P$, then $\Phi_P$
induces holomorphic connections on both $\text{ad}({\mathcal E}^0_G)$ and $\text{ad}(E_P)$.
This implies that
$$
\text{degree}(\text{ad}({\mathcal E}^0_G))\,=\, 0\, =\, \text{degree}(\text{ad}(E_P))\, ,
$$
and hence from \eqref{d1} it follows that
\begin{equation}\label{d2}
\text{degree}(\phi^*T(G/P))\,=\,0\, .
\end{equation}
Since the anti-canonical
line bundle $K^{-1}_{G/P}$ on $G/P$ is ample, from \eqref{d2} we conclude that $\phi$ is a constant map.
Consequently, $P$ is the trivial principal $P$--bundle $X\times P\, \longrightarrow\, X$.

Let
\begin{equation}\label{e17}
\varphi\, :\, {\mathcal B}_G\, \longrightarrow\, {\mathcal T}_{g,d}
\end{equation}
be the moduli space of irreducible logarithmic connections on topologically trivializable
holomorphic principal $G$--bundles. So ${\mathcal B}_G$ is the moduli space of
quadruples of the form $(X,\, D,\, E_G,\, \Phi)$, where
\begin{itemize}
\item $(X,\, D)\, \in\, {\mathcal T}_{g,d}$,

\item $E_G$ is a holomorphic principal $G$--bundle on $X$ such that $E_G$ is topologically trivial,
and

\item $\Phi$ is an irreducible logarithmic connection on $E_G$ whose
polar part is contained in $D$.
\end{itemize}
The map $\varphi$ in \eqref{e17} sends any $(X,\, D,\, E_G,\, \Phi)$ to the pair $(X,\, D)$. The moduli space 
${\mathcal B}_G$ is a smooth complex orbifold.

Any logarithmic connection on a Riemann surface $X$ is flat because $\bigwedge ^2 (TX)^*\,=\, 0$ (consequently, its
curvature $2$-form vanishes identically). So considering monodromy representation of logarithmic connections, we 
get a holomorphic map
\begin{equation}\label{e18}
\theta\, :\, {\mathcal B}_G\, \longrightarrow\, {\mathcal R}_G(S'_0)\, ,
\end{equation}
where ${\mathcal R}_G(S'_0)$ is constructed in \eqref{e16}.

We will prove that a logarithmic connection is irreducible 
if and only if the corresponding monodromy representation is irreducible.

First, let $E_G$ be a holomorphic principal $G$--bundle on $X$ equipped with a logarithmic connection $\Phi$, such 
that $\Phi$ is \textit{not} irreducible. So there is a proper parabolic subgroup $P\, \subset\, G$, a holomorphic 
reduction of the structure group of $E_G\big\vert_{X\setminus D}$ to $P$, given by a subbundle $E_P\, \subset\, 
E_G\big\vert_{X\setminus D}$, and a holomorphic connection $\Phi_P$ on $E_P$, such that the logarithmic connection 
on $E_G$ induced by $\Phi_P$ coincide with $\Phi$. Since the monodromy of $\Phi_P$ coincides with the monodromy of 
$\Phi$, the monodromy of $\Phi$ is contained in $P$, and hence
the monodromy representation for $\Phi$ is not irreducible. To prove the converse, let $\Phi$ be an 
irreducible logarithmic connection on a holomorphic principal $G$--bundle $E_G$ on $X$. Take a point $x_0\, \in\, 
X\setminus D$ and fix a point $z_0\, \in\, (E_G)_{x_0}$ in the fiber of $E_G$ over $x_0$. Taking parallel 
translations of $z_0$ along all possible homotopy classes of loops based at $x_0$ we get the monodromy 
representation
$$
H_\Phi\, :\, \pi_1(X\setminus D,\, x_0)\, \longrightarrow\, G
$$
of $\Phi$. Assume that the image of $H_\Phi$ is contained in a parabolic subgroup $P\, \subsetneq \, G$. Let 
${\mathcal S}\, \subset\, E_G\big\vert_{X\setminus D}$ be the subset obtained by taking parallel translations of 
$z_0$ along all possible homotopy classes of paths starting at $x_0$. Then
$$
E_P\, :=\, {\mathcal S}P\, \subset\, E_G\big\vert_{X\setminus D}
$$
(recall that $G$ acts on $E_G$) is a holomorphic reduction of the structure group of $E_G\big\vert_{X\setminus D}$ 
to $P$ over $X\setminus D$. The logarithmic connection $\Phi$ produces a holomorphic connection on the holomorphic 
principal $P$--bundle $E_P$, which in turn induces $\Phi$. Consequently, the logarithmic connection $\Phi$ is not 
irreducible. Thus, a logarithmic connection is irreducible if and only if the corresponding monodromy 
representation is irreducible.

\subsection{Isomonodromy}

Let
$$
d\theta\, :\, T{\mathcal B}_G\, \longrightarrow\, \theta^*T{\mathcal R}_G(S'_0)
$$
be the differential of the map $\theta$ in \eqref{e18}. The map $\theta$ is a holomorphic
submersion, meaning $d\theta$ is surjective. The kernel of $d\theta$
\begin{equation}\label{e19}
{\mathcal I}\, :=\, \text{kernel}(d\theta)\, \subset\, T{\mathcal B}_G
\end{equation}
is a holomorphic foliation on ${\mathcal B}_G$; it is known as the \textit{isomonodromy foliation}.

For any point $(X,\, D)\, \in\, {\mathcal T}_{g,d}$, the restriction of $\theta$ to
$\varphi^{-1}((X, D))$, where $\varphi$ is the projection in \eqref{e17}, is a holomorphic local diffeomorphism.
Consequently, for any point $z\, \in\, {\mathcal B}_G$, the differential of $\varphi$ 
$$
d\varphi(z)\, :\, T_z{\mathcal B}_G\, \longrightarrow\, T_{\varphi(z)}{\mathcal T}_{g,d}\, ,
$$
when restricted to the subspace ${\mathcal I}_z\, \subset\, T_z{\mathcal B}_G$ in \eqref{e19}, produces an isomorphism
$${\mathcal I}_z\, \stackrel{\sim}{\longrightarrow}\, T_{\varphi(z)}{\mathcal T}_{g,d}\, .$$ Therefore, there
is a unique holomorphic homomorphism
\begin{equation}\label{e20}
\mathbb{L}\, :\, \varphi^* T{\mathcal T}_{g,d} \, \longrightarrow\, T{\mathcal B}_G
\end{equation}
such that
\begin{itemize}
\item $d\varphi\circ\mathbb{L}\,=\, \text{Id}_{\varphi^* T{\mathcal T}_{g,d}}$, and

\item $\mathbb{L}(\varphi^* T{\mathcal T}_{g,d})\, \subset\, {\mathcal I}$, where ${\mathcal I}$
is constructed in \eqref{e19}.
\end{itemize}

Since for any point $(X,\, D)\, \in\, {\mathcal T}_{g,d}$, the restriction of $\theta$ to
$\varphi^{-1}((X, D))$ is a holomorphic local diffeomorphism, it follows
that $\mathbb{L}$ actually satisfies the condition that
\begin{equation}\label{spi}
\mathbb{L}(\varphi^* T{\mathcal T}_{g,d})\, =\, {\mathcal I}\, .
\end{equation}

\begin{proposition}\label{prop1}
Take any point $z\,=\,(X,\, D,\, E_G,\, \Phi)\, \in\, {\mathcal B}_G$.
\begin{enumerate}
\item The tangent space to ${\mathcal B}_G$ at $z$ is the first hypercohomology
$$
T_z{\mathcal B}_G\,=\, {\mathbb H}^1(X,\, {\mathcal C}_{\bullet})\, ,
$$
where ${\mathcal C}_{\bullet}$ is the complex in \eqref{e14}.

\item The homomorphism
$$
\mathbb{L}(z)\, :\, T_{\varphi(z)}{\mathcal T}_{g,d}\,=\, H^1(X,\, TX(-D))
\, \longrightarrow\, T_z{\mathcal B}_G\,=\, {\mathbb H}^1(X,\, {\mathcal C}_{\bullet})
$$
in \eqref{e20} coincides with the homomorphism of hypercohomologies
$$
\Phi^C_*\, :\, H^1(X,\, TX(-D))
\, \longrightarrow\, {\mathbb H}^1(X,\, {\mathcal C}_{\bullet})
$$
induced by the homomorphism $\Phi^C$ in \eqref{e15}.
\end{enumerate}
\end{proposition}

For the proof of Proposition \ref{prop1} the reader is referred to \cite[Proposition 3.8]{Ko}
(Proposition 3.4 of the arxiv version of \cite{Ko}), 
\cite[p.~1417, Proposition 5.1]{Ch2}, \cite{Ch1}, \cite{In} and \cite{BS}.

\section{Monodromy map on logarithmic differential systems}\label{monodromy map}

\subsection{Logarithmic differential systems}

As before, $G$ is a connected complex reductive affine algebraic group with $\dim G\, >\, 0$, and
\begin{equation}\label{ds}
d_s\, :=\, \dim_{\mathbb C}[G,\, G]\, .
\end{equation}

Consider the moduli space ${\mathcal B}_G$ in \eqref{e17}. Let
\begin{equation}\label{e21}
{\mathbb T}(G)\, \subset\, {\mathcal B}_G
\end{equation}
be the locus of all $(X,\, D,\, E_G,\, \Phi)$ such that
the holomorphic principal $G$--bundle $E_G$
on $X$ is holomorphically trivial.
Note that $E_G$ is
topologically trivial by the definition of ${\mathcal B}_G$;
also by the definition of ${\mathcal B}_G$
the logarithmic connection $\Phi$ is irreducible. The subset ${\mathbb T}(G)$ in
\eqref{e21} is a complex subspace.

\begin{proposition}\label{prop2}
The complex space ${\mathbb T}(G)$ in \eqref{e21} is a complex orbifold of dimension
$(g+d-1)\cdot\dim_{\mathbb C} G- d_s + 3g-3+d$, where $d_s$ is defined in \eqref{ds}.
\end{proposition}

\begin{proof}
Let $\varpi\, :\, C_{g,d}\, \longrightarrow\,{\mathcal T}_{g,d}$ be the universal Riemann surface equipped
with the universal divisor ${\mathcal D}\, \subset\, C_{g,d}$ of relative degree $d$ over ${\mathcal T}_{g,d}$. 
Let ${\mathcal K}\, \longrightarrow\, C_{g,d}$ be the relative holomorphic cotangent bundle for
the projection $\varpi$. 
Let $\varphi'\, :\, {\mathbb T}(G)\, \longrightarrow\,{\mathcal T}_{g,d}$ be the restriction of
the map $\varphi$ in \eqref{e17}.

For any Riemann surface $X$, the space of all logarithmic connections on the
trivial holomorphic principal $G$--bundle $X\times G \, \longrightarrow\, X$ with polar part contained
in $D\, \subset\, X$ is the vector space $H^0(X,\, K_X(D))\otimes {\mathfrak g}$, where $\mathfrak g$ is
the Lie algebra of $G$. Consequently, ${\mathbb T}(G)$ is the quotient of an open dense
subset of the total space of $\varpi_*({\mathcal K}\otimes {\mathcal O}_{C_{g,d}}({\mathcal D}))
\otimes {\mathfrak g}$ by the adjoint action of $G$; the group $G$ acts trivially on
$\varpi_*({\mathcal K}\otimes {\mathcal O}_{C_{g,d}}({\mathcal D}))$ and it has the
adjoint action on $\mathfrak g$.

Take a point $$\textbf{w}\,\in\, \varpi_*({\mathcal K}\otimes {\mathcal O}_{C_{g,d}}({\mathcal D}))
\otimes {\mathfrak g}$$ that defines an irreducible logarithmic connection on the trivial principal
$G$--bundle. The adjoint action of the center of $G$ on $\mathfrak g$ is trivial. The isotropy subgroup of
$[G,\, G]$ for the action of $[G,\, G]$ on $\textbf{w}$ is a finite subgroup of
$[G,\, G]$. The proposition follows from these.
\end{proof}

Let
\begin{equation}\label{e22}
\widehat{\theta}\, :\, {\mathbb T}(G)\, \longrightarrow\, {\mathcal R}_G(S'_0)\, ,
\end{equation}
be the restriction to ${\mathbb T}(G)\, \subset\, {\mathcal B}_G$ of the monodromy map $\theta$ in
\eqref{e18}. We are interested in the following question: When is the map $\widehat{\theta}$ an
immersion over an open dense subset of ${\mathbb T}(G)$?

\begin{remark}\label{rem-d1}
In \cite{BD} it was proved that $\widehat{\theta}$ is an immersion over an open dense subset of ${\mathbb T}(G)$, 
if $g\, \geq\, 2$, $d\,=\, 0$ and $\dim_{\mathbb C}G\, \geq\, 3$. From this it can be deduced that 
$\widehat{\theta}$ is an immersion over an open dense subset of ${\mathbb T}(G)$, if $g\, \geq\, 2$, $d\,=\,1$ and 
$\dim_{\mathbb C}G\, \geq\, 3$. To see this, first note that there is no logarithmic one-form on a compact Riemann 
surface $X$ with exactly one pole, because the residue has to be zero. So the space ${\mathbb T}(G)$ in \eqref{e22} 
for $d\,=\,1$ coincides with ${\mathbb T}(G)$ for $d\,=\, 0$. On the other hand, the natural map from the space 
${\mathcal R}_G(S'_0)$ in \eqref{e22} for $d\,=\, 0$ to the space ${\mathcal R}_G(S'_0)$ for $d\,=\, 1$ is an 
embedding; this natural map corresponds to restricting any flat $G$--connection on $S_0$ to the open subset 
$S'_0\,=\, S_0\setminus D_0$ of it. Therefore, we conclude that the map $\widehat{\theta}$ is an immersion over an 
open dense subset of ${\mathbb T}(G)$, if $g\, \geq\, 2$, $d\,=\, 1$ and $\dim_{\mathbb C}G\, \geq\, 3$.
\end{remark}

\subsection{The main theorem}

We first state a lemma of linear algebra which will be used in the proof of Theorem \ref{thm1}.

\begin{lemma}\label{lem3}
Let $\beta\, :\, V\, \longrightarrow\, W$ be a linear map between two finite dimensional
complex vector spaces. Let $S_1$ and $S_2$ be two subspaces of $V$ such that
\begin{enumerate}
\item ${\rm kernel}(\beta)\, \subset\, S_1$, and

\item the homomorphism $\beta\big\vert_{S_2}\, :\, S_2\, \longrightarrow\, W$ is injective.
\end{enumerate}
Then $\dim S_1\bigcap S_2\,=\, \dim \beta(S_1)\bigcap \beta(S_2)$.
\end{lemma}

\begin{proof}
Since $\beta\big\vert_{S_2}$ is injective, the restriction $\beta\big\vert_{S_1\cap S_2}$ is injective.
For any $v\, \in\, S_2$ with $\beta(v)\,\in\, \beta(S_1)$, there is an element $w\, \in\, S_1$
such that $\beta(v)\,=\, \beta(w)$. But then $v-w\, \in\, {\rm kernel}(\beta)\, \subset\, S_1$, and 
hence $v\, \in\, S_1$. Consequently, the restriction $\beta\big\vert_{S_1\cap S_2}$
is realized as an isomorphism between $S_1\bigcap S_2$ and $ \beta(S_1)\bigcap \beta(S_2)$.
\end{proof}

\begin{theorem}\label{thm1}
Assume that $3g-3+d\, >\, 0$ and $d\, \geq\, 1$.
The map $\widehat{\theta}$ in \eqref{e22} is an immersion over a nonempty
open dense subset of ${\mathbb T}(G)$ in the following two cases:
\begin{enumerate}
\item $g\, \geq\, 2$ and $\dim_{\mathbb C} G\, \geq \, d+2$;

\item $g\, =\, 1$ and $\dim_{\mathbb C} G\, \geq \, d$.
\end{enumerate}

The map $\widehat{\theta}$ in \eqref{e22} is nowhere an immersion in the following two cases:
\begin{enumerate}
\item $g\, =\, 0$;

\item $g\, \geq\, 1$ and $\dim_{\mathbb C}G \, <\, \frac{d+3g-3}{g}$.
\end{enumerate}
\end{theorem}

\begin{proof}
First assume that $g\, =\, 0$. The trivial holomorphic principal $G$--bundle over ${\mathbb C}{\mathbb P}^1$
is rigid \cite{Ra}, \cite{Ha}. In other words, in any holomorphic family of holomorphic
principal $G$--bundle over ${\mathbb C}{\mathbb P}^1$, parametrized by a complex manifold $Z$, the locus of
points of $Z$ over which the principal $G$--bundle on ${\mathbb C}{\mathbb P}^1$ is holomorphically
trivial is an open subset of $Z$. Therefore, the map $\widehat{\theta}$ in \eqref{e22} is nowhere an immersion.
We note that this also follows from the fact that
$$
\dim {\mathbb T}(G)- {\mathcal R}_G(S'_0)\,=\, d-3\, >\, 0
$$
if $g\,=\, 0$ (see \eqref{dr} and Proposition \ref{prop2}).

So we assume that $g\, \geq\,1$.

If $g\, \geq\, 1$ and $\dim_{\mathbb C}G \, <\, \frac{d+3g-3}{g}$, then from \eqref{dr} and Proposition \ref{prop2}
we have
$$
\dim {\mathbb T}(G)- {\mathcal R}_G(S'_0)\,=\, 3g-3+d- g\cdot\dim_{\mathbb C}G\, >\, 0\, .
$$
Hence the map $\widehat{\theta}$ in \eqref{e22} is nowhere an immersion in this case also.

So we assume that at least one of the following two holds:
\begin{enumerate}
\item $g\, \geq\, 2$ and $\dim_{\mathbb C} G\, \geq \, d+2$;

\item $g\, =\, 1$ and $\dim_{\mathbb C} G\, \geq \, d$.
\end{enumerate}

The map $\widehat{\theta}$ in \eqref{e22} is an immersion over the subset of 
${\mathbb T}(G)$ over which the homomorphism
$$
\bigwedge\nolimits^e d\widehat{\theta}\, :\, \bigwedge\nolimits^e
T{\mathbb T}(G) \,\longrightarrow\, \widehat{\theta}^*\bigwedge\nolimits^e T{\mathcal R}_G(S'_0)
$$
is fiber-wise nonzero, where $e\,=\, \dim_{\mathbb C}{\mathbb T}(G)$ and
$d\widehat{\theta}$ is the differential of the map $\widehat{\theta}$. Therefore, to prove the theorem,
it suffices to show that there is a point $z\, \in\, {\mathbb T}(G)$ such that the differential
at $z$
\begin{equation}\label{y1}
d\widehat{\theta} (z)\, :\, 
T_z{\mathbb T}(G) \,\longrightarrow\, T_{\widehat{\theta}(z)} {\mathcal R}_G(S'_0)
\end{equation}
is injective; recall that ${\mathbb T}(G)$ is irreducible.

Take a point
\begin{equation}\label{z}
z\, =\, (X,\, D,\, E_G,\, \Phi)\ \in\, {\mathbb T}(G)\, .
\end{equation}
We recall that ${\mathcal I}(z)\, =\, \text{kernel}((d\theta) (z))$
(see \eqref{e19}). The homomorphism $d\widehat{\theta} (z)$ (see \eqref{y1}) is injective
if and only if
\begin{equation}\label{e23}
{\mathcal I}(z)\cap T_z{\mathbb T}(G)\,=\, 0\, ;
\end{equation}
note that both ${\mathcal I}(z)$ and $T_z{\mathbb T}(G)$ are subspaces of the tangent space
$T_z{\mathcal B}_G$.

We will use Lemma \ref{lem3} to prove that \eqref{e23} holds when $z$ is chosen
suitably.

We recall from Proposition \ref{prop1}(1) that
$T_z{\mathcal B}_G\,=\, {\mathbb H}^1(X,\, {\mathcal C}_{\bullet})$, where
${\mathcal C}_\bullet$ is the complex in \eqref{e14}. Also, recall that the
infinitesimal deformations of the triple $(X,\, D,\, E_G)$ are
parametrized by $H^1(X,\, \text{At}(E_G)(-\log D))$. Let
\begin{equation}\label{e24}
\rho\, :\, {\mathbb H}^1(X,\, {\mathcal C}_{\bullet})\, \longrightarrow\,
H^1(X,\, \text{At}(E_G)(-\log D))
\end{equation}
be the forgetful map that sends any infinitesimal deformation of the quadruple
$$z\,=\, (X,\, D,\, E_G,\, \Phi)$$ in \eqref{z} to the infinitesimal deformation of the triple
$(X,\, D,\, E_G)$ obtained from it by simply forgetting the logarithmic connection. We shall describe $\rho$
explicitly.

Let ${\mathcal A}_\bullet$ be the one-term complex with $\text{At}(E_G)(-\log D)$ at the $0$-th position.

Consider the homomorphism ${\mathcal H}$ of complexes
\begin{equation}\label{e24b}
\begin{matrix}
{\mathcal C}_\bullet &: & \text{At}(E_G)(-\log D) & \stackrel{\widehat{\Phi}}{\longrightarrow} &
\text{ad}(E_G)\otimes K_X(D)\\
{\mathcal H}\Big\downarrow\,\,\,\,\,\,\, && \Vert &&\Big\downarrow\\
{\mathcal A}_\bullet & : & \text{At}(E_G)(-\log D) &\longrightarrow & 0
\end{matrix}
\end{equation}
where ${\mathcal C}_\bullet$ is the complex in \eqref{e14}. Let
\begin{equation}\label{e25}
{\mathcal H}_*\, :\, {\mathbb H}^1(X,\, {\mathcal C}_{\bullet})\, \longrightarrow\,
{\mathbb H}^1(X,\, {\mathcal A}_\bullet)
\,=\, H^1(X,\, \text{At}(E_G)(-\log D))
\end{equation}
be the homomorphism of hypercohomologies induced by this homomorphism of complexes. Then
the homomorphism $\rho$ in \eqref{e24} coincides with ${\mathcal H}_*$ in \eqref{e25}.

In Lemma \ref{lem3}, set $V\,=\, {\mathbb H}^1(X,\, {\mathcal C}_{\bullet})$, $W\,=\,
H^1(X,\, \text{At}(E_G)(-\log D))$, $\beta\,=\,{\mathcal H}_*$ (see \eqref{e25}), $S_1\,=\, T_z{\mathbb T}(G)$
(see \eqref{e21}) and
$S_2\,=\, {\mathcal I}(z)$ (see \eqref{e19}).

We will show that the hypotheses in Lemma \ref{lem3} are satisfied.

\begin{proposition}\label{prop3}
For the above data, the two conditions in Lemma \ref{lem3} hold.
\end{proposition}

\begin{proof}[{Proof of Proposition \ref{prop3}}]
The first condition in Lemma \ref{lem3} says that
\begin{equation}\label{ls1}
\text{kernel}({\mathcal H}_*)\, \subset\, T_z{\mathbb T}(G)\, .
\end{equation} 
To prove \eqref{ls1}, we will identify the kernel of ${\mathcal H}_*$. For this, observe that the
homomorphism of complexes ${\mathcal H}$ in \eqref{e24b} fits in the following short exact sequence of complexes:
$$
\begin{matrix}
0 && 0&& 0\\
\Big\downarrow &&\Big\downarrow &&\Big\downarrow\\
{\mathcal A}'_\bullet &: & 0 &\longrightarrow & \text{ad}(E_G)\otimes K_X(D)\\
\Big\downarrow &&\Big\downarrow&& \Vert \\
{\mathcal C}_\bullet &: & \text{At}(E_G)(-\log D) & \stackrel{\widehat{\Phi}}{\longrightarrow} &
\text{ad}(E_G)\otimes K_X(D)\\
{\mathcal H}\Big\downarrow\,\,\,\,\,\,\, && \Vert &&\Big\downarrow\\
{\mathcal A}_\bullet & : & \text{At}(E_G)(-\log D) &\longrightarrow & 0\\
\Big\downarrow&& \Big\downarrow && \Big\downarrow\\
0 && 0 && 0
\end{matrix}
$$
This short exact sequence of complexes yields the following long exact sequence of hypercohomologies:
$$
\longrightarrow\, {\mathbb H}^1(X,\, {\mathcal A}'_\bullet)\,=\,
H^0(X,\, \text{ad}(E_G)\otimes K_X(D)) \, \stackrel{\nu}{\longrightarrow}\,
{\mathbb H}^1(X,\, {\mathcal C}_{\bullet})
$$
$$
\stackrel{{\mathcal H}_*}{\longrightarrow}\, 
{\mathbb H}^1(X,\, {\mathcal A}_\bullet)\,=\,
H^1(X,\, \text{At}(E_G)(-\log D)) \, \longrightarrow \cdots \, .
$$
The above homomorphism $\nu$ corresponds to moving the holomorphic connection on the
trivializable holomorphic principal $G$--bundle $E_G$, keeping the triple $(X,\, D,\, E_G)$
fixed. This immediately implies that \eqref{ls1} holds.

The second condition in Lemma \ref{lem3} says that the restriction of the
homomorphism ${\mathcal H}_*$ to ${\mathcal I}(z)$
\begin{equation}\label{ei}
{\mathcal H}_*\big\vert_{{\mathcal I}(z)}\, :\,
{\mathcal I}(z)\, \longrightarrow\, H^1(X,\, \text{At}(E_G)(-\log D))
\end{equation}
is injective.

To prove that the homomorphism in \eqref{ei} is injective, from \eqref{spi} we conclude that
${\mathcal H}_*\big\vert_{{\mathcal I}(z)}$
is injective if the composition of homomorphisms
\begin{equation}\label{ei2}
H^1(X,\, TX(-D)) \, \stackrel{{\mathbb L}(z)}{\longrightarrow}\, {\mathbb H}^1(X,\, {\mathcal C}_{\bullet})
\, \stackrel{{\mathcal H}_*}{\longrightarrow}\, H^1(X,\, \text{At}(E_G)(-\log D))
\end{equation}
is injective, where ${\mathbb L}(z)$ is the homomorphism in \eqref{e20}. From Proposition \ref{prop1}(2)
we know that ${\mathbb L}(z)\,=\,\Phi^C_*$. Therefore, the composition of homomorphisms in \eqref{ei2}
coincides with the homomorphism of cohomologies
$$
\Phi_*\, :\, H^1(X,\, TX(-D)) \, \longrightarrow\, H^1(X,\, \text{At}(E_G)(-\log D))
$$
induced by the logarithmic connection $\Phi\, :\, TX(-D)\, \longrightarrow\, \text{At}(E_G)(-\log D)$
in \eqref{z}. But from the definition of a logarithmic connection we know that
\begin{itemize}
\item the homomorphism $\Phi$ is fiber-wise injective, and

\item $\Phi(TX(-D))$ is a direct summand of $\text{At}(E_G)(-\log D)$.
\end{itemize}
Consequently, the above homomorphism $\Phi_*$ is injective.
Hence the composition of homomorphisms in \eqref{ei2} is injective. This implies that the
homomorphism in \eqref{ei} is injective. This completes the proof of Proposition \ref{prop3}.
\end{proof}

Continuing with the proof of Theorem \ref{thm1}, in view of Proposition \ref{prop3}, from
Lemma \ref{lem3} we conclude that the statement in \eqref{e23} is equivalent to the
following statement:
\begin{equation}\label{e26}
{\mathcal H}_*({\mathcal I}(z))\cap {\mathcal H}_*(T_z{\mathbb T}(G))\,=\, 0\, ,
\end{equation}
where ${\mathcal H}_*$ is the homomorphism in \eqref{e25}.

Fix a holomorphic trivialization of the principal $G$--bundle $E_G$ in \eqref{z}. Using
it we will identify $E_G$
with the trivial holomorphic principal $G$--bundle $X\times G\, \longrightarrow\, X$. So
$\text{ad}(E_G)$ is the trivial holomorphic vector bundle $X\times
{\mathfrak g}\, \longrightarrow\, X$, where ${\mathfrak g}$ is the Lie algebra of $G$, and also
\begin{equation}\label{e27}
\text{At}(E_G)(-\log D)\,=\, \text{ad}(E_G)\oplus TX(-D)\,=\, X\times
{\mathfrak g} \oplus TX(-D)\, .
\end{equation}

Let $\Phi_0$ be the trivial logarithmic (in fact it is holomorphic) connection on the trivial 
holomorphic principal $G$--bundle $X\times G \, \longrightarrow\, X$. Note that the trivial
holomorphic connection on $E_G$ does not depend on the choice of the trivialization of $E_G$.
The homomorphism 
$$TX(-D)\, \longrightarrow\, \text{At}(E_G)(-\log D)$$ that defines $\Phi_0$ coincides with the 
inclusion map $$TX(-D)\, \hookrightarrow\, \text{ad}(E_G)\oplus TX(-D)\,=\, 
\text{At}(E_G)(-\log D)$$ (see \eqref{e27}). So we have
\begin{equation}\label{e28}
\Phi\,=\, \Phi_0+ \delta\, ,
\end{equation}
where $$\delta\, \in\, H^0(X,\, K_X(D)\otimes {\mathfrak g})\,=\, H^0(X,\, K_X(D))\otimes
{\mathfrak g}\, ;$$ recall that $\text{ad}(E_G)\,=\, X\times\mathfrak g$.

Consider the infinitesimal deformations of the triple $(X,\, D,\, E_G)$ in \eqref{z} such the
principal $G$--bundle remains trivial, but the pair $(X,\, D)$
moves. These correspond to the image of the homomorphism
$$
H^1(X,\, TX(-D))\, \longrightarrow\, H^1(X,\, \text{At}(E_G)(-\log D))\,=\,
H^1(X,\, \text{ad}(E_G))\oplus H^1(X,\, TX(-D))
$$
(see \eqref{e27} for the decomposition) given by the identity map of $H^1(X,\, TX(-D))$ and the zero map of $H^1(X,\,
TX(-D))$ to $H^1(X,\, \text{ad}(E_G))$. In other words, these correspond to the image of the
homomorphism of cohomologies
$$H^1(X,\, TX(-D))\, \longrightarrow\, H^1(X,\, \text{At}(E_G)(-\log D))$$ induced by the inclusion map
$TX(-D)\, \hookrightarrow\, \text{ad}(E_G)\oplus TX(-D)$ which is defined using \eqref{e27}.

Consequently, the subspace in \eqref{e26}
$$
{\mathcal H}_*(T_z{\mathbb T}(G))\, \subset\, H^1(X,\, \text{At}(E_G)(-\log D))\,=\,
H^1(X,\, \text{ad}(E_G))\oplus H^1(X,\, TX(-D))
$$
coincides with the subspace
$$0\oplus H^1(X,\, TX(-D))\,=\, H^1(X,\, TX(-D))\, \subset\,
H^1(X,\, \text{At}(E_G)(-\log D))
$$
$$
=\, H^1(X,\, \text{ad}(E_G))\oplus H^1(X,\, TX(-D))
\, .$$

Consider the section $\delta$ in \eqref{e28}. Using the natural duality pairing
$$TX(-D)\otimes K_X(D)\, \longrightarrow\, {\mathcal O}_X$$ it produces a homomorphism
\begin{equation}\label{e29a}
\widehat{\delta}\, :\, TX(-D)\, \, \longrightarrow\, {\mathcal O}_X\otimes {\mathfrak g}\,=\,
\text{ad}(E_G)\, .
\end{equation}
Let
\begin{equation}\label{e29}
\widehat{\delta}_*\, :\, H^1(X,\, TX(-D))\, \longrightarrow\, H^1(X,\, {\mathcal O}_X)
\otimes {\mathfrak g}\,=\, H^1(X,\, \text{ad}(E_G))
\end{equation}
be the homomorphism of cohomologies induced by $\widehat{\delta}$ in \eqref{e29a}.

We will now show that the subspace in \eqref{e26}
$$
{\mathcal H}_*({\mathcal I}(z))\, \subset\, H^1(X,\, \text{At}(E_G)(-\log D))
\,=\, H^1(X,\, \text{ad}(E_G))\oplus H^1(X,\, TX(-D))
$$
coincides with the subspace
$$
\{(\widehat{\delta}_*(v),\, v)\, \mid\, v\, \in\, H^1(X,\, TX(-D))\}\, \subset\,
H^1(X,\, \text{ad}(E_G))\oplus H^1(X,\, TX(-D))\, ,
$$
where $\widehat{\delta}_*$ is the homomorphism in \eqref{e29}.

To prove this, let
\begin{equation}\label{i}
\iota_{0*}\, :\, H^1(X,\, \text{ad}(E_G))\, \longrightarrow\, H^1(X,\, \text{At}(E_G)(-\log D))
\end{equation}
be the homomorphism of cohomologies induced by the homomorphism $\iota_0$ of sheaves in
\eqref{e6}. We note that $\iota_{0*}$ coincides with the natural map that sends any infinitesimal deformation of
$E_G$ (keeping $(X,\, D)$ fixed) to the corresponding infinitesimal deformation of $(X,\, D,\, E_G)$
where only $E_G$ is moving.

Consider the homomorphism $\Phi^C$ in \eqref{e15}
constructed from the connection $\Phi$. Let
\begin{equation}\label{z2}
\Phi^{0,C} \, :\,TX(-D)\, \longrightarrow\, {\mathcal C}^0_{\bullet}
\end{equation}
be the homomorphism as in \eqref{e15} 
constructed for the trivial connection $\Phi_0$ on $E_G$; here ${\mathcal C}^0_{\bullet}$ is the
complex as in \eqref{e14} for the trivial connection $\Phi_0$. From \eqref{e28} it follows
immediately that
\begin{equation}\label{e30}
{\mathcal H}\circ\Phi^C- {\mathcal H}^0\circ\Phi^{0,C}\,=\, \iota_{0}\circ\widehat{\delta}\, ,
\end{equation}
where $\widehat{\delta}$, ${\mathcal H}$ and $\iota_{0}$ are the homomorphisms in
\eqref{e29a}, \eqref{e24b} and \eqref{e6} respectively, while
${\mathcal H}^0$ is the homomorphism for the trivial holomorphic connection $\Phi_0$ constructed
as in \eqref{e24b} (by substituting $\Phi_0$ in place of $\Phi$ in the
construction of ${\mathcal H}$). As in Proposition \ref{prop1}(2), let
\begin{equation}\label{e32}
\Phi^{0,C}_* \, :\, H^1(X,\, TX(-D))
\, \longrightarrow\, {\mathbb H}^1(X,\, {\mathcal C}^0_{\bullet})
\end{equation}
be the homomorphism of hypercohomologies induced by $\Phi^{0,C}$ in \eqref{z2}. From \eqref{e30}
we conclude that
\begin{equation}\label{e31}
{\mathcal H}_*\circ\Phi^C_* - {\mathcal H}^0_*\circ\Phi^{0,C}_*\,=\, \iota_{0*}\circ \widehat{\delta}_*\, ,
\end{equation}
where $\widehat{\delta}_*$, ${\mathcal H}_*$, $\Phi^{0,C}_*$, $\iota_{0*}$ and $\Phi^C_*$
are the homomorphisms in \eqref{e29},
\eqref{e25}, \eqref{e32}, \eqref{i} and Proposition \ref{prop1}(2) respectively, and
\begin{equation}\label{nh1}
{\mathcal H}^0_*\, :\,
{\mathbb H}^1(X,\, {\mathcal C}^0_{\bullet})\, \longrightarrow\, H^1(X,\, \text{At}(E_G)(-\log D))
\end{equation}
is the homomorphism of hypercohomologies induced by the homomorphism ${\mathcal H}^0$ in \eqref{e30}.
Note that both sides of \eqref{e31} are actually
homomorphisms from $H^1(X,\, TX(-D))$ to 
$H^1(X,\, \text{At}(E_G)(-\log D)).$ Also, note that from
the decomposition in \eqref{e27} it follows immediately that the homomorphism
$\iota_{0*}$ in \eqref{e31} is injective. In fact, the decomposition in \eqref{e27} realizes
$H^1(X,\, \text{ad}(E_G))$ as a direct summand of $H^1(X,\, \text{At}(E_G)(-\log D))$.

Now consider the homomorphism $$\mathbb{L}(z)\, :\, T_{(X,D)}{\mathcal T}_{g,d} \, \longrightarrow\,
T_z {\mathcal B}_G$$ in \eqref{e20} constructed for the connection $\Phi$ in the expression of $z$ in
\eqref{z}. Let
$$
\mathbb{L}^0\, :\, T_{(X,D)}{\mathcal T}_{g,d} \, \longrightarrow\, T_{(X,D,E_G,\Phi_0)}{\mathcal B}_G
$$
be the homomorphism as in \eqref{e20} constructed for the trivial connection $\Phi_0$. From
Proposition \ref{prop1}(2) we know that
\begin{equation}\label{z3}
{\mathcal H}_*\circ\mathbb{L} - {\mathcal H}^0_*\circ \mathbb{L}^0\,=\,
{\mathcal H}_*\circ\Phi^C_* -{\mathcal H}^0_*\circ\Phi^{0,C}_*\, ,
\end{equation}
where ${\mathcal H}_*$ and ${\mathcal H}^0_*$ are the homomorphisms in \eqref{e25} and \eqref{nh1}
respectively;
recall that $T_z {\mathcal B}_G\,=\, {\mathbb H}^1(X,\, {\mathcal C}_{\bullet})$ and
$T_{(X,D,E_G,\Phi_0)}{\mathcal B}_G\,=\, {\mathbb H}^1(X,\, {\mathcal C}^0_{\bullet})$.

Combining \eqref{e31} and \eqref{z3} it follows that
\begin{equation}\label{e33}
{\mathcal H}_*\circ\mathbb{L} - {\mathcal H}^0_*\circ \mathbb{L}^0\,=\, \iota_{0*}\circ \widehat{\delta}_*\, ,
\end{equation}
where $\widehat{\delta}_*$ is the homomorphism in \eqref{e29}.
It was noted earlier that the decomposition in \eqref{e27} realizes
$H^1(X,\, \text{ad}(E_G))$ as a direct summand of $H^1(X,\, \text{At}(E_G)(-\log D))$.

Therefore, from \eqref{e33} and \eqref{spi} we conclude that the subspace in \eqref{e26}
$$
{\mathcal H}_*({\mathcal I}(z))\, \subset\, H^1(X,\, \text{At}(E_G)(-\log D))
\,=\, H^1(X,\, \text{ad}(E_G))\oplus H^1(X,\, TX(-D))
$$
(see \eqref{e27} for the above decomposition) coincides with the subspace
$$
\{(\widehat{\delta}_*(v),\, v)\, \mid\, v\, \in\, H^1(X,\, TX(-D))\}\, \subset\,
H^1(X,\, \text{ad}(E_G))\oplus H^1(X,\, TX(-D))\, .
$$
On the other hand, it was shown earlier the subspace in \eqref{e26}
$$
{\mathcal H}_*(T_z{\mathbb T}(G))\, \subset\, H^1(X,\, \text{At}(E_G)(-\log D))\,=\,
H^1(X,\, \text{ad}(E_G))\oplus H^1(X,\, TX(-D))
$$
coincides with the subspace
$$0\oplus H^1(X,\, TX(-D))\,=\, H^1(X,\, TX(-D))\, \subset\,
H^1(X,\, \text{ad}(E_G))\oplus H^1(X,\, TX(-D))\, .$$ Combining these two
we obtain an isomorphism
\begin{equation}\label{e34}
\eta\, :\, \text{kernel}(\widehat{\delta}_*)
\, \stackrel{\sim}{\longrightarrow}\,
{\mathcal H}_*({\mathcal I}(z))\cap{\mathcal H}_*(T_z{\mathbb T}(G))
\end{equation}
that sends any $v\, \in\, \text{kernel}(\widehat{\delta}_*)\, \subset\, H^1(X,\, TX(-D))$ to
$$
(0,\, v)\, \in\, H^1(X,\, \text{ad}(E_G))\oplus H^1(X,\, TX(-D))\,=\,
H^1(X,\, \text{At}(E_G)(-\log D))\, .
$$

Consequently, \eqref{e26} holds if and only if we have
\begin{equation}\label{e26b}
\text{kernel}(\widehat{\delta}_*)\,=\,0\, ,
\end{equation}
where $\widehat{\delta}_*$ is the homomorphism constructed in \eqref{e29}.

Take any subspace
$$
V\, \subset\, H^0(X,\, K_X(D))\, .
$$
Let $H^1(X,\, TX(-D))\otimes V\, \longrightarrow\, H^1(X,\, {\mathcal O}_X)$ be the
homomorphism constructed using the duality pairing $TX(-D)\otimes K_X(D)\, \longrightarrow\,
{\mathcal O}_X$. Let
\begin{equation}\label{q1}
F_V\, :\, H^1(X,\, TX(-D))\, \longrightarrow\, H^1(X,\, {\mathcal O}_X)\otimes V^*
\end{equation}
be the homomorphism given by it. From the construction of $\widehat{\delta}_*$ in
\eqref{e29} we see that
$$
\text{kernel}(\widehat{\delta}_*)\,=\,\text{kernel}(F_V)\, ,
$$
where $V\, \subset\, H^0(X,\, K_X(D))$ is the image of the homomorphism
\begin{equation}\label{q3}
H_\delta\, :\, {\mathfrak g}^*\, \longrightarrow\, H^0(X,\, K_X(D))
\end{equation}
given by $\delta$ in \eqref{e28}; note that since
$\delta\, \in\, H^0(X,\, K_X(D))\otimes {\mathfrak g}$, it produces a
homomorphism $H_\delta$ as in \eqref{q3} by sending any $w\,\in\, {\mathfrak g}^*$
to $w(\delta)\, \in\, H^0(X,\, K_X(D))$. Consequently, \eqref{e26b} holds if
and only if
\begin{equation}\label{q2}
\text{kernel}(F_{H_\delta({\mathfrak g}^*)})\,=\, 0\, ,
\end{equation}
where $H_\delta$ and $F_{H_\delta({\mathfrak g}^*)}$ are the homomorphism constructed in \eqref{q3}
and \eqref{q1} respectively.

It is evident that there is an element
$z\, =\, (X,\, D,\, E_G,\, \Phi)\ \in\, {\mathbb T}(G)$ such that \eqref{q2} holds if and
only if there is a subspace $V\, \subset\, H^0(X,\, K_X(D))$, with $\dim V\,\leq\, \dim \mathfrak g$,
satisfying the condition that the homomorphism $F_V$ in \eqref{q1} is injective. Indeed, choosing a
homomorphism $$\delta'\, :\, {\mathfrak g}^*\, \longrightarrow\, H^0(X,\, K_X(D))$$ for which
$V\, \subset\, \delta'({\mathfrak g}^*)$, consider the element $\delta\, \in\, H^0(X,\, K_X(D))\otimes
{\mathfrak g}$ given by $\delta'$. Then the
logarithmic connection $(X,\, D,\, X\times G, \, \Phi_0+\delta)$ satisfies \eqref{q2}, where $\Phi_0$ is
the trivial holomorphic connection on $X\times G\, \longrightarrow\, X$.

First assume that $g\, =\, 1$ (hence, by hypothesis, $d\, \geq \, 1$) and $\dim_{\mathbb C} G\, \geq \, d$. This 
implies that
$$
\dim H^0(X,\, K_X(D))\,=\, d\,\leq\, \dim_{\mathbb C} G\, .
$$
So in this case there is a subspace $V\, \subset\, H^0(X,\, K_X(D))$, with $\dim V\,
\leq\, \dim \mathfrak g$, for which the homomorphism $F_V$ in \eqref{q1} is injective,
if the homomorphism
\begin{equation}\label{j1}
F_{H^0(X,K_X(D))}\, :\, H^1(X,\, TX(-D))\, \longrightarrow\, H^1(X,\, {\mathcal O}_X)\otimes
H^0(X,\, K_X(D))^*
\end{equation}
is injective; if the homomorphism in \eqref{j1} is injective, then we may take
$V$ to be $H^0(X,\, K_X(D))$ itself and the homomorphism $F_V$ is injective.

The homomorphism in \eqref{j1} is injective if the dual homomorphism
\begin{equation}\label{j2}
F^*_{H^0(X,K_X(D))}\, :\, H^0(X,\, K_X)\otimes H^0(X,\, K_X(D))\, \longrightarrow\,
H^0(X,\, K^{\otimes 2}_X(D))
\end{equation}
is surjective. Now the homomorphism in \eqref{j2} is injective because 
$\dim H^0(X,\, K_X)\,=\, 1$. On the other hand, we have
$$
\dim H^0(X,\, K_X(D))\,=\, d\,=\, \dim H^0(X,\, K^{\otimes 2}_X(D))\, ,
$$
so the homomorphism in \eqref{j2} is an isomorphism, in particular, it is
surjective. This proves the theorem when
$g\, =\, 1$ and $\dim_{\mathbb C} G\, \geq \, d$.

Now assume that $g\, \geq\, 2$ and $\dim_{\mathbb C} G\, \geq \, d+2$.

Since $\dim\mathfrak g\, \geq\, d+2$, we conclude that there is an element
$z\, =\, (X,\, D,\, E_G,\, \Phi)\ \in\, {\mathbb T}(G)$ such that \eqref{q2} holds
if there is a subspace $V\, \subset\, H^0(X,\, K_X(D))$, with $\dim V\,
\leq\, d+2$, for which the homomorphism $F_V$ in \eqref{q1} is injective. From
Lemma \ref{lem4} and Lemma \ref{lem5} (see also Remark \ref{rem3}) it follows that such a
subspace $V$ exists. This completes the proof of the theorem.
\end{proof}

\begin{remark}\label{gd0}
{}From Theorem \ref{thm1} it follows that when $g\,=\,1$ and $d\,=\, 0$, the map
$\widehat{\theta}$ in \eqref{e22} is an immersion over a nonempty
open dense subset of ${\mathbb T}(G)$. Indeed, from Remark \ref{rem-d1} we know that
${\mathbb T}(G)$ for $d\,=\,0$ coincides with ${\mathbb T}(G)$ for $d\,=\,1$. On the other hand,
${\mathcal R}_G(S'_0)$ for $d\,=\,0$ is embedded into ${\mathcal R}_G(S'_0)$ for $d\,=\,1$. From
Theorem \ref{thm1} we know that the map
$\widehat{\theta}$ in \eqref{e22} is an immersion over a nonempty
open dense subset of ${\mathbb T}(G)$ if $g\,=\,1$ and $d\,=\, 1$. Therefore, the same holds when
$g\,=\,1$ and $d\,=\, 0$. Recall that $\dim G\, >\, 0$.
\end{remark}

In view of Remark \ref{rem-d1}, we assume that $d\, >\, 1$ when $g\, >\, 1$.

\begin{lemma}\label{lem4}
Take integers $g\, >\, 1$ and $d\, >\, 1$. Then for any compact connected non--hyperelliptic Riemann
surface $X$ of genus $g\, \geq\, 3$, and any effective divisor $D$ on $X$ of degree $d$, there exists a subspace
$W\, \subset\, H^0(X,\, K_X(D))$, with $\dim W\,=\,d+2$,
such that the homomorphism constructed in \eqref{q1}
$$F_W\, :\, H^1(X,\, TX(-D))\, \longrightarrow\, H^1(X,\, {\mathcal O}_X)\otimes W^*$$
is injective.
\end{lemma}

\begin{proof}
For a compact Riemann surface $X$ of genus $g$, and an effective divisor $D$ on $X$ of
degree $d$, denote the holomorphic line bundle
$K^{\otimes 2}_X\otimes {\mathcal O}_X(D))$ by $K^2_X(D)$.
For any subspace $V\, \subset\, H^0(X,\, K_X(D))$, let
$$
F^*_V\, :\, H^0(X,\, K_X)\otimes V \, \longrightarrow\, K^0(X,\, K^2_X(D))
$$
be the dual of the homomorphism $F_V$ in \eqref{q1}.

We need to show that there is a $W$ with $\dim W\,=\,d+2$ such that the above homomorphism
\begin{equation}\label{q4}
F^*_W\, :\, H^0(X,\, K_X)\otimes W\, \longrightarrow\, H^0(X,\, K^2_X(D))
\end{equation}
is surjective.

Consider the natural homomorphism
\begin{equation}\label{q5}
J\, :\, H^0(X,\, K_X)\otimes H^0(X,\, K_X(D))\, \longrightarrow\, H^0(X,\, K^2_X(D))\, .
\end{equation}

We will now show that under our assumptions, the homomorphism $J$ in \eqref{q5} is surjective. To this end, we 
apply \cite[Theorem (4.e.1)]{Gr} and see that it suffices to prove that
\begin{equation}\label{hv}
h^1(X,{\mathcal O}_X(D))\,\le\, g - 2\, .
\end{equation}

When $D$ is non--special, \eqref{hv} evidently holds. So
we suppose that $D$ is special. In order to prove \eqref{hv}, first assume that $d \,\geq\, 4$.
Then Clifford's theorem (see \cite[p.~251]{GH}) says that
$h^0(X,\,{\mathcal O}_X(D)) \,\le\, d/2 +1$. Now using Riemann-Roch theorem we get that
$$d + 1 - g \,=\, h^0(X,\,{\mathcal O}_X(D)) - h^1(X,\, {\mathcal O}_X(D)) \,\le\,
\frac{d}{2} + 1 - h^1(X,\,{\mathcal O}_X(D))\, .$$
This implies that \eqref{hv} holds, and hence $J$ is surjective in this case by \cite[Theorem (4.e.1)]{Gr}.

Assume now that $d\,=\,2$ or $d\,=\,3$. Since $X$ is not hyperelliptic, if $d\,=\,2$, then we have $h^0(X, \,{\mathcal 
O}_X(D))\,=\, 1$. If $d\,=\,3$, Clifford's theorem implies that $h^0(X,\,{\mathcal O}_X(D))\,\le \, 2$. Then the 
Riemann-Roch theorem implies that \eqref{hv} holds in both these cases. Applying \cite[Theorem (4.e.1)]{Gr}, we 
infer that $J$ is surjective in these cases as well.

Consequently, we have obtained the surjectivity of the map $J$ in \eqref{q5} for any pair $(X,\,D)$ as in the lemma.
 
{}From the commutative diagram
\[
\xymatrix
{
 H^0(X,\,K_X)\otimes H^0(X,\,K_X(D))\ar[r]^{J} \ar[d]&H^0(X,\,K_X^2(D)) \ar[d]\\
H^0(X,\,K_X)\otimes H^0(X,\,K_X(D))/H^0(X,\,K_X)\ar[r]&H^0(X,\,K_X^2(D))/H^0(X,\,K_X^2)
}
\]
we notice that the surjectivity of $J$ implies the surjectivity of the map
\[
H^0(X,\,K_X)\otimes (H^0(X,\,K_X(D))/H^0(X,\,K_X))\,\longrightarrow\, H^0(X,\,K_X^2(D))/H^0(X,\,K_X^2)\,.
\]

Consider $U\,\subset\, H^0(X,\,K_X(D))$ of dimension $(d-1)$ such that $U\cap H^0(X,\,K_X)\,=\,
\{0\}$ inside $H^0(X,\,K_X(D))$. Then the map
\[
U\,\longrightarrow\, H^0(X,\,K_X(D))/H^0(X,\,K_X)
\]
is an isomorphism and hence the induced map
\begin{equation}\label{eqn:U}
H^0(X,\,K_X)\otimes U\,\longrightarrow\, H^0(X,\,K_X^2(D))/H^0(X,\,K_X^2)
\end{equation}
is surjective.

On the other hand, since $X$ is non--hyperelliptic, \cite[Theorem 1.1]{Gi} (whose proof is attributed to
Lazarsfeld) shows that for a general subspace $W_0\,\subset\, H^0(X,\,K_X)$ of dimension 3 the multiplication map
\begin{equation}\label{eqn:W0}
H^0(X,\,K_X)\otimes W_0\,\longrightarrow\, H^0(X,\,K_X^2)
\end{equation}
is surjective.
Set $$W\,=\,W_0\oplus U\,\subset\, H^0(X,\,K_X(D))\, .$$ The surjectivity of the maps in \eqref{eqn:U})
and \eqref{eqn:W0} implies the surjectivity of
\[
F^*W:H^0(X,\,K_X)\otimes W\,\longrightarrow\, H^0(X,\,K_X^2(D))
\]
which concludes the proof.
\end{proof}

\begin{remark}
Lemma \ref{lem4} is optimal in the following sense. If $W\subset H^0(X,\,K_X(D))$ is a subspace such that the 
intersection $W\cap H^0(X,\,K_X)$ inside $H^0(X,\,K_X(D))$ is at least three--dimensional and $H^0(X,\,K_X)\otimes W
\,\longrightarrow\, H^0(X,\,K_X^2(D))$ is surjective, then $\dim W\,\ge\, d+2$. This claim is easily
obtained by reverting the argument in the proof of Lemma \ref{lem4}.
\end{remark} 
 
Lemma \ref{lem4} excluded the case of $g\,=\,2$.
This is dealt with separately below. 
 
\begin{lemma}\label{lem5}
Let $X$ be a compact connected Riemann surface of genus two, and let $D$
be an effective divisor of degree $d\,>\,1$ such that $D\not\,\in\, |K_X|$. Then the multiplication map
\[
H^0(X,\,K_X)\otimes H^0(X,\, K_X(D))\,\longrightarrow\, H^0(X,\,K_X^2(D))
\]
is surjective.
\end{lemma}
 
\begin{proof}
We start with the short exact sequence
\[
0\,\longrightarrow\, TX\,\longrightarrow\, H^0(X,\,K_X)\otimes \mathcal{O}_X\,\longrightarrow\, K_X
\,\longrightarrow\, 0\, ,
\]
twist it by $K_X(D)$ and take the corresponding long exact sequence of cohomologies
$$
H^0(X,\,K_X)\otimes H^0(X,\, K_X(D)) \,\longrightarrow\, H^0(X,\,K_X^2(D)) \,\longrightarrow\,
H^1(X,\,\mathcal{O}_X(D)) \,\longrightarrow\,\, .
$$
By the hypothesis, we have $H^1(X,\,\mathcal{O}_X(D))\,=\, 0$ and hence from this exact
sequence of cohomologies it follows that the multiplication map is surjective.
\end{proof}
 
\begin{remark}\label{rem3}
Note that, under the hypotheses of Lemma \ref{lem5}, the Riemann-Roch theorem implies that $h^0(X,
\,K_X(D))\,=\,d+1\, \,<\, d+2$.
\end{remark}

\section*{Acknowledgement}

We thank the referee for a comment on irreducibility of logarithmic connections that led to a clarification. We 
thank Frank Loray for an useful discussion. We thank Michiaki Inaba for pointing out \cite{Ko}. MA was partly 
supported by the CNCS - UEFISCDI project PN-III-P4-ID-PCE- 2020-0029. IB is partially supported by a J. C. Bose 
Fellowship. IB and SD was partially supported by the French government through the UCAJEDI Investments in the 
Future project managed by the National Research Agency (ANR) with the reference number ANR2152IDEX201. SH is 
supported by the DFG grant HE 6829/3-1 of the DFG priority program SPP 2026 {\em Geometry at 
Infinity}.

%%%%%%%%%%%%%%%%%%%%%%%%%%%%%%%%%%%%%%%%%%%%%%%%%%%%%%%%%%%%%%%%%


\begin{thebibliography}{ZZZZZ}

\bibitem[At]{At} M. F. Atiyah, Complex analytic connections in fibre
bundles, \textit{Trans. Amer. Math. Soc.} \textbf{85} (1957), 181--207.

\bibitem[BD]{BD} I. Biswas and S. Dumitrescu, The monodromy map
from differential systems to character variety is generically immersive,
arxiv.org/abs/2002.05927.

\bibitem[BHH]{BHH} I. Biswas, V. Heu and J. Hurtubise, Isomonodromic deformations of
logarithmic connections and stability, {\it Math. Ann.} {\bf 366} (2016), 121--140.

\bibitem[BS]{BS} C. L. Bremer and D. S. Sage, Isomonodromic deformations of connections
with singularities of parahoric formal type, {\it Comm. Math. Phys.} {\bf 313} (2012), 175--208.

\bibitem[CDHL]{CDHL} G. Calsamiglia, B. Deroin, V. Heu and F Loray, The Riemann-Hilbert mapping
for ${\mathfrak s}{\mathfrak l}_2$ systems over genus two curves, {\it Bull. Soc. Math. France}
{\bf 147} (2019), 159--195.

\bibitem[Ch1]{Ch1} T. Chen, The associated map of the nonabelian Gauss--Manin connection, Thesis
(Ph.D.)--University of Pennsylvania (2012),
https://www.math.upenn.edu/grad/dissertations/ChenThesis.pdf.

\bibitem[Ch2]{Ch2} T. Chen, The associated map of the nonabelian Gauss--Manin connection,
{\it Cent. Eur. Jour. Math.} {\bf 10} (2012), 1407--1421.

\bibitem[De]{De} P. Deligne, \textit{Equations diff\'erentielles \`a points
singuliers r\'eguliers}, Lecture Notes in Mathematics, Vol. 163, Springer-Verlag,
Berlin-New York, 1970.

\bibitem[Do]{Do} I. F. Donin, Construction of a versal family of deformations for
holomorphic bundles over a compact complex space, {\it Math. USSR Sb.} {\bf 23}
(1974), 405--416.

\bibitem[Gi]{Gi} D. Gieseker, A lattice version of the KP equation,
{\it Acta Math.} {\bf 168} (1992), 219--248. 

\bibitem[Go]{Go} W. M. Goldman, The symplectic nature of fundamental groups of surfaces,
\textit{Adv. in Math.} {\bf 54} (1984), 200--225.

\bibitem[Gr]{Gr} M. L. Green, Koszul cohomology and geometry of projective varieties, {\it J. Differential 
Geometry} {\bf 19} (1984), 125--171.

\bibitem[GH]{GH} P. Griffiths and J. Harris, \textit{Principles of Algebraic Geometry}, Pure and Applied 
Mathematics, Wiley-Interscience, New York, 1978.

\bibitem[Ha]{Ha} G. Harder, Halbeinfache {G}ruppenschemata \"{u}ber vollst\"{a}ndigen {K}urven,
{\it Invent. Math.} {\bf 6} (1968), 107--149.

\bibitem[Hu]{Hu} L. Huang, On joint moduli spaces, {\it Math. Ann.} {\bf 302} (1995), 61--79.

\bibitem[In]{In} M. Inaba, Moduli of parabolic connections on a curve and
Riemann-Hilbert correspondence, {\it Jour. Alg. Geom.} {\bf 22} (2013), 407--480.

\bibitem[Ka]{Ka} N. M. Katz, An overview of Deligne's work on Hilbert's twenty-first problem,
{\it Mathematical developments arising from Hilbert problems}
(Proc. Sympos. Pure Math., Vol. XXVIII, Northern Illinois Univ., 
De Kalb, Ill., 1974), pp. 537--557. Amer. Math. Soc., Providence, R. I., 1976.

\bibitem[Ko]{Ko} A. Komyo, Hamiltonian structures of isomonodromic deformations on moduli spaces of parabolic 
connections, {\it Jour. Math. Soc. Japan} (to appear), arXiv:1611.03601.

\bibitem[Ra]{Ra} A. Ramanathan, Deformations of principal bundles on the projective line,
{\it Invent. Math.} {\bf 71} (1983), 165--191.

\bibitem[Si]{Si} A. S. Sikora, Character varieties, {\it Trans. Amer. Math. Soc.}
{\bf 364} (2012), 173--208.

\end{thebibliography}
\end{document}